\documentclass[oneside]{amsart}
\usepackage[a4paper]{geometry}
\usepackage{amsmath}
\usepackage{amsthm}
\usepackage{amsfonts}
\usepackage{graphics}

\usepackage{amssymb}
\usepackage{verbatim}
\usepackage{amscd}
\usepackage{enumerate}
\usepackage{graphicx}
\usepackage{siunitx}
\usepackage{color}
\numberwithin{equation}{section}

\usepackage{mathabx}
\usepackage{mathtools}
\usepackage[tableposition=top]{caption}
\usepackage{booktabs,dcolumn}

\theoremstyle{plain}

\newtheorem{theorem}{Theorem}[section]
\newtheorem{proposition}[theorem]{Proposition}

\newtheorem{lemma}[theorem]{Lemma}
\newtheorem{corollary}[theorem]{Corollary}

\theoremstyle{definition}

\usepackage{lipsum}

\usepackage{appendix}

\usepackage[numbers]{natbib}

\usepackage{hyperref}

\def\Om{\Omega}

\def\C{\mathbb C}
\allowdisplaybreaks

\usepackage{color}

\begin{document}

\title {Bergman-Toeplitz operators on fat Hartogs triangles}


\email{tkhanh@uow.edu.au \, jiakunl@uow.edu.au \, ttp754@uowmail.edu.au }
  
\author[Tran Vu Khanh, Jiakun Liu, Phung Trong Thuc]{Tran Vu Khanh, Jiakun Liu, Phung Trong Thuc}

\address
	{Institute for Mathematics and its Applications, School of Mathematics and Applied Statistics,
	University of Wollongong,
	Wollongong, NSW 2522, AUSTRALIA.}	
 \thanks{Khanh was supported by ARC grant
 	DE160100173; Liu was supported by ARC grant DP170100929; Thuc was supported by PhD scholarship in ARC grant DE140101366}	

	
\begin{abstract}
In this paper, we obtain some $L^{p}$ mapping properties of the Bergman-Toeplitz operator
\[
f\longrightarrow T_{K^{-\alpha}}\left(f\right):=\intop_{\Omega}K_{\Omega}\left(\cdot,w\right)K^{-\alpha}\left(w,w\right)f\left(w\right)dV(w)
\]
on fat Hartogs triangles $\Omega_{k}:=\left\{ \left(z_{1},z_{2}\right)\in\mathbb{C}^{2}:\left|z_{1}\right|^{k}<\left|z_{2}\right|<1\right\} $, 
where $\alpha\in\mathbb{R}$ and $k\in \mathbb Z^+$.
\end{abstract}

\maketitle
\section{Introduction}
In this paper, we continue our study of  the ``gain" $L^p$-estimate properties of Bergman-Toeplitz operators on pseudoconvex domains following \cite{KLT17}. Let $\Om\subset\C^n$ be a bounded domain, and let $A^{p}\left(\Omega\right)$ be the closed subspace of holomorphic functions in $L^{p}\left(\Omega\right)$. 
Given a measurable function $\psi$ on $\Omega$, the Bergman-Toeplitz operator with symbol $\psi$ is defined by 
\begin{equation}
f\longrightarrow T_{\psi}\left(f\right)\left(z\right):=\intop_{\Omega}K\left(z,w\right)\psi\left(w\right)f\left(w\right)dV(w),
\end{equation}
where  $K(\cdot, \cdot)$ is the Bergman kernel associated to $\Om$.
Recently, we proved in \cite{KLT17} that for a large class of weakly pseudoconvex
smooth domains in $\mathbb{C}^{n}$, the Bergman-Toeplitz operator $T_{\psi}$ with $\psi(z)=K^{-\alpha}(z,z)$ maps from $L^{p}\left(\Omega\right)$ to $A^{q}\left(\Omega\right)$ continuously if and only if $\alpha\geq\frac{1}{p}-\frac{1}{q}$, for any $1<p\le q<\infty$. 
As a corollary, we have the associated Bergman projection is self $L^p$ bounded for any $p\in (1,+\infty)$, and moreover, that is sharp in the sense that the Bergman projection is not bounded from $L^p$ to $A^q$ for any $q>p$, (see also \cite{McNSt94,PhSt77,McN89,ChDu06}).\\

As a non-smooth case, the fat Hartogs triangle $\Om_k\subset \C^2$ is defined by 
\begin{equation}\label{intHa}
\Omega_{k}:=\left\{ \left(z_{1},z_{2}\right)\in\mathbb{C}^{2} : \left|z_{1}\right|^{k}<\left|z_{2}\right|<1\right\},\qquad\mbox{for }  k\in \mathbb Z^+.
\end{equation}
The model of Hartogs triangles and their variants has recently attracted
particular attention through the study of several problems in complex
analysis; see e.g. \cite{ChL17,EdMcN16,ChZe16,ChSh13,HaZe17}.  The fat Hartogs triangle $\Om_k$ is a pseudoconvex domain but not a hyperconvex domain (see e.g. \cite{ZWO99,CCW99} for the characterisation of pseudoconvexity and hyperconvexity of Reinhardt domains). 
Nevertheless, the Bergman kernel of $\Om_k$ can be computed explicitly thanks to the work of Edholm \cite{Edh16}. This important fact provides useful estimates on the Bergman kernel and then the self $L^p$ boundedness of the Bergman projection. In particular, Edholm and McNeal \cite{EdMcN17} proved that the Bergman projection associated to $\Om_k$ is self $L^p$-bounded  if and only if $p\in (\frac{2k+2}{k+2},\frac{2k+2}{k})$. This generalises the result in the case $k=1$ by Chakrabarti and Zeytuncu \cite{ChZe16}. It should be interesting to add that the Bergman projections associated to the
Hartogs triangle domains $\Omega_{\gamma}:=\left\{ \left(z_{1},z_{2}\right)\in\mathbb{C}^{2}:\left|z_{1}\right|^{\gamma}<\left|z_{2}\right|<1\right\} $
with $\gamma>0,\gamma\notin\mathbb{Q}$  are $L^p$ bounded if and only if $p=2$, see \cite{EdMcN17}. \\

It is reasonable to expect that for Hartogs triangles $\Omega_{k}$, the Bergman projections cannot gain the $L^{p}$ regularity.
Moreover, it is also of particular interest to obtain a holomorphic function with higher regularity from an input function in the $L^{p}$ space. Motivated by this, in this paper we study the Bergman-Toeplitz operators $T_{K^{-\alpha}}$, where $K$ is the Bergman kernel on the diagonal and $\alpha\in\mathbb{R}$. We shall prove that for  $1<p\le q<\infty$ such that $\frac{k+2}{2k}-\frac{1}{kp}>\frac{1}{q}>\frac{k}{2k+2}$, the Bergman-Toeplitz operator $T_{K^{-\alpha}}$ is $L^p\text-L^q$ bounded if and only if $\alpha\ge \frac{1}{p}-\frac{1}{q}$.  It is natural to show that if $q\ge \frac{2k+2}{k}$, then $T_{K^{-\alpha}}$ cannot be bounded from $L^{p}\left(\Omega_{k}\right)$
to $A^{q}\left(\Omega_{k}\right)$ for any $\alpha\in\mathbb{R}$. Surprisingly, for sufficiently small $p,q$, indeed, $\frac{1}{q}+\frac{1}{kp}\geq\frac{k+2}{2k}$,   we shall point out that  there exists $\alpha>0$ such that $T_{K^{-\alpha}}$ is still bounded from $L^p(\Om_k)$ to $A^q(\Om_k)$.   The precise statement of our main result is as follows:

\begin{theorem}\label{thmIn1}
	For $k\in \mathbb Z^+$, let $\Omega_{k}$ be the Hartogs triangle domain  defined by \eqref{intHa} and let $1<p\leq q<\infty$. Then we have the following conclusions:
	\begin{enumerate}[(i)]
		\item If $\frac{k}{2k+2}\geq\frac{1}{q}$, then there is no $\alpha\in\mathbb{R}$ such that $T_{K^{-\alpha}}:L^p(\Omega_{k})\to A^q(\Omega_{k})$ continuously.
		\item If $\frac{k+2}{2k}-\frac{1}{kp}>\frac{1}{q}>\frac{k}{2k+2}$ then  $T_{K^{-\alpha}}:L^p(\Omega_{k})\to A^q(\Omega_{k})$ continuously if and only if 
		$\alpha\ge \frac{1}{p}-\frac{1}{q}$.
		\item If $\frac{1}{q}\geq\frac{k+2}{2k}-\frac{1}{kp}$ then  $T_{K^{-\alpha}}:L^p(\Omega_{k})\to A^q(\Omega_{k})$ continuously if and only if 
		$\alpha> \frac{1}{p}-\left(\frac{k+2}{2k}-\frac1{kp}\right)$.
	\end{enumerate}
\end{theorem}
The proof of this theorem is divided into four small theorems below. In \S2, Theorem~\ref{thm 1} and~\ref{thm 2}, we provide the proof of part $(i)$ and the sufficient conditions of $(ii)$ and $(iii)$ in Theorem~\ref{thmIn1}. In \S3, Theorem~\ref{thm 3} and~\ref{thm 4}, we prove the necessary conditions of $(ii)$ and $(iii)$ in Theorem~\ref{thmIn1}.  The following corollary is a direct consequence of Theorem \ref{thmIn1}.

\begin{corollary}\label{cor1} Let $P$ be the Bergman projection associated to $\Om_k$. Then
	\begin{enumerate}[(i)]
		\item $P$ maps  from $L^p(\Om_k)$ to $A^p(\Om_k)$ continuously if and only if $p\in (\frac{2k+2}{k+2},\frac{2k+2}{k})$; and
		\item $P$ cannot map from $L^{p}\left(\Omega_{k}\right)$
		to $A^{q}\left(\Omega_{k}\right)$ continuously if  $1<p<q<\infty$. 
		\end{enumerate}
	
\end{corollary}
We remark that Corollary~\ref{cor1} $(i)$ has been proved  in  \cite[Theorem 1.2]{EdMcN16}. Finally, throughout the paper we use  $\lesssim$ and $\gtrsim$  to denote inequalities up to a positive multiplicative constant; and $\approx$ for the combination of $\lesssim$ and $\gtrsim$.

\section{Sufficient conditions}
We first recall the basic properties of the Bergman kernel associated to $\Om_k$. The Bergman kernel of $\Om_k$ can be computed explicitly as (see \cite{EdMcN16})
\[
K\left(z,w\right)=\dfrac{p_{k}\left(s\right)t^{2}+q_{k}\left(s\right)t+s^{k}p_{k}\left(s\right)}{k\pi^{2}\left(1-t\right)^{2}\left(t-s^{k}\right)^{2}}
\]
for $z=(z_1,z_2), w=(w_1,w_2)\in \Om_k$, where $s:=z_{1}\overline{w_{1}}$, $t:=z_{2}\overline{w_{2}}$, 
\[
p_{k}\left(s\right):=\sum_{j=1}^{k-1}j\left(k-j\right)s^{j-1}\text{ and }\;q_{k}\left(s\right):=\sum_{j=1}^{k}\left(j^{2}+\left(k-j\right)^{2}s^{k}\right)s^{j-1}.
\]
Here, we use the convention $\sum_{j=1}^{k-1}\cdots :=0$ if $k=1$. Since $|s^k|<|t|<1$, the Bergman kernel has the upper bound 
\begin{equation}
\left|K\left(z,w\right)\right|\lesssim\dfrac{\left|z_2\bar w_2\right|}{\left|1-z_2\bar w_2\right|^{2}\left|z_2\bar w_2-(z_1\bar w_1)^{k}\right|^{2}}.\label{eq:Pluq1}
\end{equation}
Combining the upper bound with the fact that $p_k(s)\ge 0 $ and $q_k(s)\ge 1$ for $s\in\mathbb R^+$, 
we obtain the ``sharp" estimate of the Bergman kernel on the diagonal
\begin{equation}
K\left(z,z\right)\thickapprox\dfrac{|z_2|^2}{\left(1-\left|z_{2}\right|^2\right)^{2}\left(\left|z_{2}\right|^2-\left|z_{1}\right|^{2k}\right)^{2}}\label{eq:Pluq2}
\end{equation}
for  $z=(z_1,z_2)\in \Om_k$.\\

Theorem~\ref{thmIn1} $(i)$ is covered in the following theorem. 
\begin{theorem}\label{thm 1} Let $q\ge \frac{2k+2}{k}$. Then $T_{K^{-\alpha}}$ cannot map from $L^p(\Om_k)$ to $A^q(\Om_k)$ continuously, for any $p> 0$ and $\alpha\in\mathbb R$. 
\end{theorem}
\begin{proof}
Note that if $T_{K^{-\alpha_{0}}}$ is bounded from $L^{p}\left(\Omega_{k}\right)$
to $A^{q}\left(\Omega_{k}\right)$ for some $\alpha_{0}<0$ then $T_{K^{-\alpha}}$
maps from $L^{p}\left(\Omega_{k}\right)$ to $A^{q}\left(\Omega_{k}\right)$
continuously for any $\alpha\geq0$. Therefore, we may assume $\alpha\geq0$. In order to prove Theorem \ref{thm 1}, we shall show that $T_{K^{-\alpha}}(\bar z_2)=\frac{c}{z_2}$ for some non-zero constant $c$. Then the conclusion follows since $\bar z_2\in L^p(\Om_k)$ for any $p>0$, but $\frac{1}{z_2}\not\in L^q(\Om_k)$ if $q\ge \frac{2k+2}{k}$ (by a simple calculation). 

Now, we recall the fact (see \cite{EdMcN16}) 
\[
\left\{ z^{\beta}:\beta\in\mathcal{B}:=\left\{ \left(\beta_{1},\beta_{2}\right)\in\mathbb{Z}^{2}:\beta_{1}\geq0,\beta_{1}+k\left(\beta_{2}+1\right)>-1\right\} \right\} 
\]
forms an orthogonal basis for $A^{2}\left(\Omega_{k}\right)$. It
suffices to prove $\left\langle K^{-\alpha} \bar z_2,z^{\beta}\right\rangle =0$
for any $\beta\in\mathcal{B}$, apart from $\beta=\left(0,-1\right)$. To see this, observe that the function $\psi\left(z\right):=K^{-\alpha}\left(z,z\right)$
can be represented as $\psi\left(z\right)=\Phi\left(\left|z_{1}\right|,\left|z_{2}\right|\right)$,
for a bounded function $\Phi:\mathbb{R}\times\mathbb{R}\rightarrow\mathbb{R^{+}}$.
Therefore 
\begin{eqnarray*}
	\left\langle K^{-\alpha} \bar z_2,z^{\beta}\right\rangle  & = & \intop_{\Omega_{k}}\Phi\left(\left|z_{1}\right|,\left|z_{2}\right|\right)\overline{z_{2}}\overline{z_{1}^{\beta_{1}}z_{2}^{\beta_{2}}}dz\\
	& = & \left(\intop_{0}^{2\pi}e^{-\imath\theta_{1}\beta_{1}}d\theta_{1}\right)\left(\intop_{0}^{2\pi}e^{-\imath\theta_{2}\left(\beta_{2}+1\right)}d\theta_{2}\right)\left(\intop_{U}\Phi\left(r_{1},r_{2}\right)r_{1}^{\beta_{1}+1}r_{2}^{\beta_{2}+2}dr_1dr_2\right)\\
	& = & 0,
\end{eqnarray*}
unless $\beta=\left(0,-1\right)$. Here $U:=\left\{ \left(r_{1},r_{2}\right):0\leq r_{1},r_{1}^{k}<r_{2}<1\right\} $. This completes the proof of Theorem \ref{thm 1}.
\end{proof}

The next theorem is the main goal of this section, in which we prove  the  
sufficient conditions of $(ii)$ and $(iii)$ in Theorem~\ref{thmIn1}.   
\begin{theorem}\label{thm 2} Let $1<p\le q<2+\frac{2}{k}$. If $p$, $q$ and $\alpha$ satisfy either 
	\begin{enumerate}[(i)]
	\item $\frac{1}{q}<\frac{k+2}{2k}-\frac{1}{kp}$ and 
	$\alpha\ge \frac{1}{p}-\frac{1}{q}$, or 
	\item $\frac{1}{q}\geq\frac{k+2}{2k}-\frac{1}{kp}$ and  
	$\alpha> \frac{1}{p}-\left(\frac{k+2}{2k}-\frac1{kp}\right)$
\end{enumerate}
then   $T_{K^{-\alpha}}$ maps from $L^p(\Omega_{k})$ to $A^q(\Omega_{k})$ continuously.
	\end{theorem}

One of the fundamental  tools to establish the self $L^p$ boundedness of the Bergman projection is Schur's test lemma (see \cite{McNSt94}, \cite{EdMcN16}). In our recent work \cite{KLT17}, Schur's test lemma has been generalised for studying the $L^{p}$-$L^{q}$ mapping property of Toeplitz operators.

\begin{lemma}\label{Pro-thmIn3}
\cite[Theorem 5.1]{KLT17}
Let $\left(X,\mu\right)$ and $\left(Y,\nu\right)$ be measure spaces
with $\sigma$-finite, positive measures; let $1<p\leq q<\infty$
and $\eta\in\mathbb{R}.$ Let $K:X\times Y\rightarrow\mathbb{C}$
and $\psi:Y\rightarrow\mathbb{C}$ be measurable functions. Assume
that there exist positive measurable functions $h_{1},h_{2}$ on $Y$
and $g$ on $X$ such that 
\[
h_{1}^{-1}h_{2}\psi\in L^{\infty}\left(Y,d\nu\right)
\]
and the inequalities 
\begin{eqnarray}\label{S est} 
\intop_{Y}\left|K\left(x,y\right)\right|^{\eta p'}h_{1}^{p'}\left(y\right)d\nu\left(y\right)  &\leq&  C_{1}g^{p'}\left(x\right), \nonumber \\
\intop_{X}\left|K\left(x,y\right)\right|^{\left(1-\eta\right)q}g^{q}\left(x\right)d\mu\left(x\right)  &\leq&  C_{2}h_{2}^{q}\left(y\right), 
\end{eqnarray}
hold for almost every $x\in\left(X,\mu\right)$ and $y\in\left(Y,\nu\right),$
where $\frac{1}{p}+\frac{1}{p'}=1$ and $C_{1},C_{2}$ are positive
constants.

Then, the Toeplitz operator $T_{\psi}$ associated to the kernel $K$
and the symbol $\psi$ defined by 
\[
\left(T_{\psi}f\right)\left(x\right):=\intop_{Y}K\left(x,y\right)f\left(y\right)\psi\left(y\right)d\nu\left(y\right),
\]
is bounded from $L^{p}\left(Y,\nu\right)$ onto $L^{q}\left(X,\mu\right).$
Furthermore,
\[
\left\Vert T_{\psi}\right\Vert _{L^{p}\left(Y,\nu\right)\rightarrow L^{q}\left(X,\mu\right)}\leq C_{1}^{\frac{p-1}{p}}C_{2}^{\frac{1}{q}}\left\Vert h_{1}^{-1}h_{2}\psi\right\Vert _{L^{\infty}\left(Y,\nu\right)}.
\]

\end{lemma}

In order to employ Lemma~\ref{Pro-thmIn3}, we first establish integral estimates on the Bergman kernel of the fat Hartogs triangle $\Om_k$.   

\begin{proposition}\label{Pro-thmIn1}
	Let $a,b,c$ be real numbers satisfying  
\begin{equation}\label{m1}
a\ge 1, \quad -1<b<0\quad \text{ and }\quad  -a+2b+c+\frac{2}{k}>-2.
\end{equation}
	Then for any $z\in\Omega_{k},$ 
	\begin{eqnarray}
	\intop_{\Omega_{k}}\left|K\left(z,w\right)\right|^{a}|r(w)|^b|w_2|^c dV(w) & \lesssim & K^{a-1}\left(z,z\right)|r(z)|^b |z_2|^{a-2b-2 },  \label{eq:Le31}
	\end{eqnarray}
	where $r\left(z\right):=\left(\left|z_{2}\right|^{2}-\left|z_{1}\right|^{2k}\right)\left(\left|z_{2}\right|^{2}-1\right)$.
\end{proposition}

\begin{proof}\label{Proof of Pro-thmIn1}
Set $J(z):=\intop_{\Omega_{k}}\left|K\left(z,w\right)\right|^{a}|r\left(w\right)|^b|w_2|^cdV(w)$. By \eqref{eq:Pluq1}, it follows
 
	\begin{eqnarray*}
		J(z) & \lesssim & \intop_{\Omega_{k}}\dfrac{\left|z_{2}\overline{w_{2}}\right|^{a}\left(\left|w_{2}\right|^{2}-\left|w_{1}\right|^{2k}\right)^{b}\left(1-\left|w_{2}\right|^{2}\right)^{b}|w_2|^c}{\left|1-z_{2}\overline{w_{2}}\right|^{2a}\left|z_{2}\overline{w_{2}}-z_{1}^{k}\overline{w_{1}}^{k}\right|^{2a}}dV(w)\\
		&= & \intop_{D\setminus \{0\}}\dfrac{\left|z_{2}\right|^a\left|w_2\right|^{a+c}\left(1-\left|w_{2}\right|^{2}\right)^{b}}{\left|1-z_{2}\overline{w_{2}}\right|^{2a}}\left[\intop_{D(|w_2|^{\frac1k})}\dfrac{\left(\left|w_{2}\right|^{2}-\left|w_{1}\right|^{2k}\right)^{b}}{\left|z_{2}\overline{w_{2}}-z_{1}^{k}\overline{w_{1}}^{k}\right|^{2a}}dV(w_{1})\right]dV(w_{2}),
	\end{eqnarray*}
	where $D$ is the open unit disk in $\C$ 
	and $D(|w_2|^{\frac{1}{k}}):=\left\{ w_{1}\in\mathbb{C}:\left|w_{1}\right|<\left|w_{2}\right|^{\frac1k}\right\}$. By the change of variable $u:=\frac{w_1^k}{w_2}$, the expression in the bracket $\left[\;\right]$ can be rewritten as 
	\begin{equation}\label{b1}
	\begin{split}
		[\cdots] = & \left|w_{2}\right|^{-2a+2b}\left|z_{2}\right|^{-2a}\intop_{D(|w_2|^{\frac{1}{k}})}\left|1-\dfrac{z_{1}^{k}\overline{w_{1}}^{k}}{z_{2}\overline{w_{2}}}\right|^{-2a}\left(1-\left|\dfrac{w_{1}^{k}}{w_{2}}\right|^{2}\right)^{b}dV(w_{1})\\
	 =& \frac{1}{k^2}\left|w_{2}\right|^{-2a+2b+\frac{2}{k}}\left|z_{2}\right|^{-2a}\intop_{D}\left|1-\dfrac{\overline{z_{1}}^{k}}{\overline{z_{2}}} u\right|^{-2a}\left(1-\left|u\right|^{2}\right)^{b}\left|u\right|^{\frac2k-2}dV(u).
	 \end{split}\end{equation}
By Lemma~\ref{lmm1} below, the integral term in the last line of \eqref{b1} is dominated by $\left(1-\left|\frac{z_1^k}{z_2}\right|^2\right)^{-2a+b+2}$. Thus, \eqref{b1} continues as 
	\[
	[\cdots]\lesssim\left|w_{2}\right|^{-2a+2b+\frac2k}\left|z_{2}\right|^{-2a}\left(1-\left|\dfrac{z_{1}^{k}}{z_{2}}\right|^{2}\right)^{-2a+b+2}.
	\]
Therefore, 
	\begin{equation*}\label{a1}
	\begin{split}
		J(z) \lesssim & \left|z_{2}\right|^{3a-2b-4}\left(\left|z_{2}\right|^{2}-\left|z_{1}\right|^{2k}\right)^{-2a+b+2}\intop_{D\setminus\{0\}}\left|1-w_{2}\overline{z_{2}}\right|^{-2a}\left(1-\left|w_{2}\right|^{2}\right)^{b}\left|w_{2}\right|^{-a+2b+c+\frac{2}{k}}dV(w_{2})\\
		 \lesssim & \left|z_{2}\right|^{3a-2b-4}\left(\left|z_{2}\right|^{2}-\left|z_{1}\right|^{2k}\right)^{-2a+b+2}\left(1-\left|z_{2}\right|^{2}\right)^{-2a+b+2}\\
		\lesssim & K^{a-1}\left(z,z\right) |r(z)|^b\left|z_{2}\right|^{a-2b-2},
\end{split}	\end{equation*}
where the second inequality follows by  using Lemma~\ref{lmm1} again since $-a+2b+c+\frac{2}{k}>-2$; and the last one follows by \eqref{eq:Pluq2}.
\end{proof}
The proof of Proposition \ref{Pro-thmIn1} is complete but we have skipped a crucial technical point that we face now. 
\begin{lemma}\label{lmm1}
	Let $a$, $b$ and $c$ be real numbers such that 
$$a\ge 1,\quad 	-1<b<0\quad  \text{and }\quad  c>-2.$$
	Then,  for any $v\in D$, 
	\begin{equation}\label{Prof-thmIn2-1}
	I_{a,b,c}(v):=\intop_{D}\left|1-u\bar v \right|^{-2a}\left(1-\left|u\right|^{2}\right)^{b} \left|u\right|^{c }dV(u)\lesssim\left(1-\left|v\right|^{2}\right)^{-2a+b+2}.
	\end{equation}
\end{lemma}
\begin{proof}
This lemma  is a slight extension of \cite[Lemma 3.2]{EdMcN16},  in which the estimate 
\begin{equation}
\label{cc1}
I_{1,b,c}(v)\lesssim (1-|v|^2)^{b}
\end{equation}
has been proved for $-1<b<0$ and $-2<c\le 0$. The estimate \eqref{cc1} can be automatically extended to the case $c>-2$. Now,  if $a\ge 1$ then 
$$|1-u\bar v|^{2-2a}\le (1-|v|)^{2-2a}\lesssim (1-|v|^2)^{2-2a};$$
and hence by \eqref{cc1}
$$ I_{a,b,c}(v)\lesssim(1-|v|^2)^{2-2a} I_{1,b,c}(v)\lesssim (1-|v|^2)^{-2a+b+2},
$$
provided $-1<b<0$ and $c>-2$. 
\end{proof}

Now we are ready to prove Theorem~\ref{thm 2}.

\begin{proof}[Proof of Theorem~\ref{thm 2}] Let $p'=\frac{p}{p-1}$, $0<\beta<\min\{\frac{1}{q}, \frac{1}{p'}\}$ and $\gamma<(1+\frac{2}{k})(1-\frac{1}{p})$. Combining the choice of $p', \beta,\gamma$ and the hypothesis  $1<p\le q<2+\frac{2}{k}$, it is clear that the relations  
	$$a\ge 1,\quad -1<b<0\quad \text{and}\quad  -a+2b+c+\frac{2}{k}>-2$$ satisfy 
	for both choices   $$(a,b,c)=(1,-\beta p',(2\beta-\gamma)p')\quad\text{and}\quad (a,b,c)=\left(\frac{q}{p},-\beta q, (2\beta -\frac{1}{p'})q\right).$$ 
	Thus, by Proposition~\ref{Pro-thmIn1}, the following integral estimates hold
 \begin{eqnarray} \intop_{\Omega_{k}}\left|K\left(z,w\right)\right||r(w)|^{-\beta p'}|w_2|^{(2\beta-\gamma)p'}dV(w) & \lesssim&  |r(z)|^{-\beta p'}|z_2|^{2\beta p'-1},\label{eq:MaT1}\\ \intop_{\Omega_{k}}\left|K\left(z,w\right)\right|^{\frac{q}{p}}|r(z)|^{-\beta q}|z_2|^{\frac{(2\beta p'-1)q}{p'}}dV(z) & \lesssim & K^{\frac{q}{p}-1}\left(w,w\right)|r(w)|^{-\beta q}|w_2|^{\frac{q}{p}+2\beta q-2}.\label{eq:MaT3} \end{eqnarray}
 
 This can be read that the integral estimates in Lemma~\ref{Pro-thmIn3} hold with $\eta=\frac{1}{p'}$,  $$h_1(w):=|r(w)|^{-\beta}|w_2|^{2\beta-\gamma},\quad g(z):=|r(z)|^{-\beta}|z_2|^{2\beta-\frac{1}{p'}}$$
  and $$h_2(w):=K^{\frac{1}{p}-\frac{1}{q}}(w,w)|r(w)|^{-\beta}|w_2|^{\frac{1}{p}+2\beta-\frac{2}{q}}.$$
In order to conclude that the Bergman-Toeplitz operator $T_{K^{-\alpha}}:L^p(\Om_k)\to A^q(\Om_k)$ continuously, we shall choose $\gamma<\left(1+\frac{2}{k}\right)\left(1-\frac{1}{p}\right)$ such that 
$$\mathcal A(w):=h_1^{-1}(w)h_2(w)K^{-\alpha}(w,w)=K^{\frac1p-\frac1q-\alpha}(w,w) |w_2|^{\gamma+\frac{1}{p}-\frac{2}{q}}$$
is uniformly bounded for all $w\in\Om_k$. \\
 
If $\frac{1}{q}<\frac{k+2}{2k}-\frac1{kp}$ and $\alpha\ge \frac{1}{p}-\frac{1}{q}$, by choosing $\gamma=\frac2q-\frac1p$, we have $\mathcal A(w)\le 1$ for all $w\in \Om_k$. Moreover, with this choice, the requirement $\gamma<(1+\frac{2}{k})(1-\frac{1}{p})$ is equivalent to the given condition $\frac{1}{q}<\frac{k+2}{2k}-\frac1{kp}$. This proves the first part of  Theorem~\ref{thm 2}.

If $\frac{1}{q}\ge \frac{k+2}{2k}-\frac{1}{kp}$ and $\alpha>\frac{1}{p}-\left( \frac{k+2}{2k}-\frac{1}{kp}\right)$ then 
$$\mathcal A(w)\lesssim |w_2|^{-\frac2p+\frac2q+2\alpha+\gamma+\frac1p-\frac2q}=|w_2|^{2\alpha+\gamma-\frac1p}.$$
Here, we have used the fact that $\alpha>\frac{1}{p}-\left( \frac{k+2}{2k}-\frac{1}{kp}\right)\ge \frac{1}{p}-\frac1q$ and  $K^{-1}(w,w)\lesssim |w_2|^2$ by \eqref{eq:Pluq2}. 
Now by choosing $\gamma=-2\alpha+\frac1p$, we have $\mathcal A$ is uniformly bounded. The proof is complete since the requirement  $\gamma<(1+\frac{2}{k})(1-\frac1p)$ is equivalent to  $\alpha>\frac{1}{p}-\left( \frac{k+2}{2k}-\frac{1}{kp}\right)$.

\end{proof}

\section{Necessary conditions}

To complete the proof of Theorem \ref{thmIn1}, the remaining task is to show the respective lower bounds for $\alpha$. That is, if $T_{K^{-\alpha}}$ maps from $L^{p}\left(\Omega_{k}\right)$ to $A^{q}\left(\Omega_{k}\right)$ continuously, then $\alpha$ must be greater than or
equal to the desired values. 
We shall prove parts $(ii)$ and $(iii)$ in Theorem \ref{thmIn1} by using two different approaches, respectively.
We remark that the
underlying idea of both arguments is to derive the desirable property
from singular points. \\

In the first approach, we illustrate a technique involving the use of the pluricomplex Green function. This method may be applied to a more general context. We first recall that 
for  a bounded pseudoconvex domain  $\Omega\subset \mathbb{C}^{n}$, 
the pluricomplex Green function with a pole $w\in\Omega$ is defined
by 
\[
G\left(\cdot,w\right):=\sup\left\{ u\left(\cdot\right):u\in PSH^{-}\left(\Omega\right),\limsup_{z\rightarrow w}\left(u\left(z\right)-\log\left|z-w\right|\right)<\infty\right\}. 
\]
Here $PSH^{-}\left(\Omega\right)$ denotes the set of all negative plurisubharmonic functions in $\Omega$. The relation between the pluricomplex Green function and the Bergman kernel has been studied by several authors, see e.g. \cite{Her99,Blo05,ChFu11,Blo14}. One of the most important facts from these results is the very weak assumption on the regularity of the domain.  We refer the reader to \cite[Chapter 6]{Kli91} for an introduction of the pluricomplex Green function.\\

The following proposition is first proved by Herbort \cite{Her99} with the constant on the right hand side depending on the diameter of the domain,  and B{\l}ocki \cite{Blok14} improved it to the sharp one as follows.

\begin{proposition}[Herbort-B{\l}ocki]\label{HER}
	Let $\Omega$ be a bounded pseudoconvex domain in $\mathbb C^n$ and let $t$ be any positive
	number. Then for any holomorphic function $f$ on $\Omega$ and any $w\in\Omega$, 
	\[
	\intop_{\left\{ G\left(\cdot,w\right)<-t\right\} }\left|f\left(z\right)\right|^{2}dz\geq e^{-2nt}\dfrac{\left|f\left(w\right)\right|^{2}}{K\left(w,w\right)}.
	\]
\end{proposition}

The next lemma provides a relation between the Bergman kernel and the pole $w$ on the sublevel set $\left\{ G\left(\cdot,w\right)<-1\right\} $ for the Hartogs triangles $\Omega_{k}$.

\begin{lemma}\label{Plu1}
	Let $w\in\Omega_{k}$,
	\[
	K\left(z,z\right)\left|z_{2}\right|^{2}\thickapprox K\left(w,w\right)\left|w_{2}\right|^{2}
	\]
	for any $z\in\left\{ G\left(\cdot,w\right)<-1\right\} $.
\end{lemma}

\begin{proof}
	We first prove the following elementary fact.\\
{\bf Claim}: {\it   
	If $a,b\in D$ and 
	$\left|\dfrac{a-b}{1-a\overline{b}}\right|<\dfrac{1}{e}$
then $ 1-\left|a\right|^2\approx 1-\left|b\right|^2$.}
\begin{proof}[Proof of the claim] Set $\xi=\frac{b-a}{1-a\bar b}$. Then $a=\frac{b-\xi}{1-\xi\bar b}$ and hence $\frac{1-|a|^2}{1-|b|^2}=\frac{1-|\xi|^2}{|1-\xi \bar b|^2}$. This implies the stated claim by the fact that 
	$$\frac{e-1}{e+1}<\frac{1-|\xi|}{1+|\xi|}\le \frac{1-|\xi|^2}{|1-\xi \bar b|^2}\le \frac{1+|\xi|}{1-|\xi|}<\frac{e+1}{e-1},\quad \text{since $b\in D$}. $$
\end{proof}
We now proceed the proof of Lemma \ref{Plu1}. 
	Recall that, see e.g. \cite{Kli95}, for $z=\left(z_{1},z_{2}\right)$ and $w=\left(w_{1},w_{2}\right)$,
	\[
	G_{D\times D}\left(z,w\right)=\max\left\{ \log\left|\dfrac{z_{1}-w_{1}}{1-z_{1}\overline{w_{1}}}\right|,\log\left|\dfrac{z_{2}-w_{2}}{1-z_{2}\overline{w_{2}}}\right|\right\} .
	\]
	
	Since the map 
	\begin{eqnarray*}
		F:\Omega_{k} & \longrightarrow & D\times D\\
		\left(z_{1},z_{2}\right) & \longrightarrow & \left(\dfrac{z_{1}^{k}}{z_{2}},z_{2}\right)
	\end{eqnarray*}
	is holomorphic, we obtain 
	\[
	G_{D\times D}\left(F\left(z_{1},z_{2}\right),F\left(w_{1},w_{2}\right)\right)\leq G_{\Omega_{k}}\left(\left(z_{1},z_{2}\right),\left(w_{1},w_{2}\right)\right)
	\]
	for any $z,w\in\Omega_{k}$. Thus for any $z\in\left\{ z\in\Omega_{k}:G\left(z,w\right)<-1\right\} $,
	we have
$$
	\left|\dfrac{\frac{z_{1}^{k}}{z_{2}}-\frac{w_{1}^{k}}{w_{2}}}{1-\frac{z_{1}^{k}\overline{w_{1}}^{k}}{z_{2}\overline{w_{2}}}}\right|<\dfrac{1}{e}\text{ and }\left|\dfrac{z_{2}-w_{2}}{1-z_{2}\overline{w_{2}}}\right|<\dfrac{1}{e}.$$
Using the above claim, it follows 
	\begin{equation}
	1-|z_2|^2\approx 1-|w_2|^2\quad \text{and }\quad 1-\left|\frac{z_1^k}{z_2}\right|^2\approx 1-\left|\frac{w_1^k}{w_2}\right|^2. \label{eq:Pluq4}
\end{equation}
	Now, the desired result can be obtained by  \eqref{eq:Pluq4} and the fact $$K(z,z)|z_2|^2\approx \left[(1-|z_2|^2)\left(1-\left|\frac{z_1^k}{z_2}\right|^2\right)\right]^{-2}.$$
\end{proof}

We now turn to the proof of the necessary condition in Theorem~\ref{thmIn1} $(ii)$. 
\begin{theorem}\label{thm 3}
	Let $1<p\leq q<2+\frac{2}{k}$ and let $\alpha\in\mathbb{R}$. Assume that $T_{K^{-\alpha}}:L^{p}\left(\Omega_{k}\right)\to A^{q}\left(\Omega_{k}\right)$ continuously.  Then $\alpha\geq\frac{1}{p}-\frac{1}{q}$.
\end{theorem}
\begin{proof} 
First, we may assume that $\alpha<1$, otherwise it implies the conclusion. Since $K(w,z)w_2$ is holomorphic in $w$, one has
\begin{eqnarray*}
\intop_{\Omega_{k}}\left|K\left(z,w\right)\right|^{2}K^{-\alpha}\left(w,w\right)\left|w_{2}\right|^{2}dw & = & \hspace*{-0.15cm}\intop_{\Omega_{k}}K\left(z,w\right)K^{-\alpha}\left(w,w\right)\overline{w_{2}}K\left(w,z\right)w_{2}dw\\
 & = & \hspace*{-0.15cm}\intop_{\Omega_{k}}K\left(z,w\right)K^{-\alpha}\left(w,w\right)\overline{w_{2}}\left(\intop_{\Omega_{k}}K\left(w,\xi\right)K\left(\xi,z\right)\xi_{2}d\xi\right)dw\\
 & = & \hspace*{-0.15cm}\intop_{\Omega_{k}}\left(\intop_{\Omega_{k}}K\left(w,\xi\right)K^{-\alpha}\left(w,w\right)K\left(z,w\right)\overline{w_{2}}dw\right)K\left(\xi,z\right)\xi_{2}d\xi.
\end{eqnarray*}
By H{\"o}lder's inequality and the boundedness of $T_{K^{-\alpha}}$, it continues as
\begin{eqnarray}
\intop_{\Omega_{k}}\left|K\left(z,w\right)\right|^{2}K^{-\alpha}\left(w,w\right)\left|w_{2}\right|^{2}dw & \leq & \left\Vert \intop_{\Omega_{k}}K\left(w,\cdot\right)K^{-\alpha}\left(w,w\right)K\left(z,w\right)\overline{w_{2}}dw\right\Vert _{L^{q}}\hspace*{-0.15cm}\left\Vert K\left(z,\cdot\right)\left(\cdot\right)_{2}\right\Vert _{L^{q'}}\nonumber \\
 & \lesssim & \left\Vert K\left(z,\cdot\right)\left(\cdot\right)_{2}\right\Vert _{L^{p}\left(\Omega_{k}\right)}\left\Vert K\left(z,\cdot\right)\left(\cdot\right)_{2}\right\Vert _{L^{q'}\left(\Omega_{k}\right)},\label{eq:pluri1}
\end{eqnarray}
where $\frac{1}{q}+\frac{1}{q'}=1$. 
By Lemma \ref{Pro-thmIn1}, 
\begin{eqnarray}
\left\Vert K\left(z,\cdot\right)\left(\cdot\right)_{2}\right\Vert _{L^{p}\left(\Omega_{k}\right)}\left\Vert K\left(z,\cdot\right)\left(\cdot\right)_{2}\right\Vert _{L^{q'}\left(\Omega_{k}\right)} & \lesssim & \left|z_{2}\right|^{2-\frac{2}{p}-\frac{2}{q'}}K^{2-\frac{1}{p}-\frac{1}{q'}}\left(z,z\right)\label{eq:Plute2}\\
 & = & |z_{2}|^{\frac{2}{q}-\frac{2}{p}}K^{1-\frac{1}{p}+\frac{1}{q}}(z,z).\nonumber 
\end{eqnarray}

Note that the appearance of the term $w_{2}$ above allows us to apply Lemma \ref{Pro-thmIn1}. On the other hand, the LHS of \eqref{eq:pluri1} satisfies
\begin{eqnarray*}
\intop_{\Omega_{k}}\left|K\left(z,w\right)\right|^{2}K^{-\alpha}\left(w,w\right)\left|w_{2}\right|^{2}dw & \geq & \intop_{\left\{ G\left(\cdot,z\right)<-1\right\} }\left|K\left(z,w\right)\right|^{2}K^{-\alpha}\left(w,w\right)\left|w_{2}\right|^{2}dw\\
 & = & \intop_{\left\{ G\left(\cdot,z\right)<-1\right\} }\left|K\left(z,w\right)\right|^{2}\left|w_{2}\right|^{2+2\alpha}\left(K\left(w,w\right)\left|w_{2}\right|^{2}\right)^{-\alpha}dw\\
 & \gtrsim & \left(K\left(z,z\right)\left|z_{2}\right|^{2}\right)^{-\alpha}\intop_{\left\{ G\left(\cdot,z\right)<-1\right\} }\left|K\left(z,w\right)w^{2}_{2}\right|^{2}dw\\
 & \gtrsim & \left(K\left(z,z\right)\left|z_{2}\right|^{2}\right)^{-\alpha}K\left(z,z\right)\left|z_{2}\right|^{4}.\\
\end{eqnarray*}
Here we have used Lemma \ref{Plu1} and Proposition \ref{HER}.
From this, \eqref{eq:pluri1} and \eqref{eq:Plute2}, we get $\alpha\geq\frac{1}{p}-\frac{1}{q}$
by letting $z_{2}\rightarrow1$.
This completes the proof of Theorem \ref{thm 3}.

\end{proof}

The second approach is to construct an appropriate sequence $\left\{ f_{j}\right\} $
and establish the lower bound from the hypothesis 
\begin{equation}\label{eq:Sha1}
\sup_{j}\dfrac{\left\Vert T_{K^{-\alpha}}\left(f_{j}\right)\right\Vert _{L^{q}\left(\Omega_{k}\right)}}{\left\Vert f_{j}\right\Vert _{L^{p}\left(\Omega_{k}\right)}}<\infty.
\end{equation}
This approach is standard and depends quite heavily on the intrinsic
information of our domains $\Omega_{k}.$ 
Let us use this approach to prove  the necessary condition in Theorem~\ref{thmIn1} $(iii)$. 
\begin{theorem}\label{thm 4}
	Let $1<p\leq q<2+\frac{2}{k}$ and let $\alpha\geq0$.  Assume that  $T_{K^{-\alpha}}:L^{p}\left(\Omega_{k}\right)\to A^{q}\left(\Omega_{k}\right)$ continuously.  Then $\alpha>\frac{k+1}{kp}-\frac{k+2}{2k}$.
\end{theorem} 
\begin{proof}
	We employ a similar computation as in \cite{ChL17}. Define
	the sequence $\left\{ f_{j}\right\} _{j=1}^{\infty}$ by 
	\[
	f_{j}\left(z\right):=\begin{cases}
	h\left(\left|z_{2}\right|\right)\overline{z_{2}} & \text{; }\qquad a_{j+1}<\left|z_{2}\right|<1,\\
	0 & \text{; \qquad elsewhere, }
	\end{cases}
	\]
	where $a_{j}:=\frac{1}{j^{j}}$ and the function $h:\left(0,1\right] \rightarrow\left(0,\infty\right)$
	is defined by 
	\[
	h\left(x\right):=x^{\frac{1}{l}-1-\left(2+\frac{2}{k}\right)\frac{1}{p}}\text{ for }\;x\in\left(a_{l+1},a_{l}\right];\quad l=1,2,{\scriptstyle \ldots}.
	\]
	We now can easily check that 
	\begin{eqnarray*}
\left\Vert f_{j}\right\Vert _{L^{p}\left(\Omega_{k}\right)}^{p} & \lesssim & \intop_{a_{j+1}}^{1}h^{p}\left(r_{2}\right)r_{2}^{p+\frac{2}{k}+1}dr_{2}\lesssim\sum_{l=1}^{\infty}\left(\intop_{a_{l+1}}^{a_{l}}h^{p}\left(r_{2}\right)r_{2}^{p+\frac{2}{k}+1}dr_{2}\right)\\
 & \lesssim & \sum_{l=1}^{\infty}l\left(l^{-p}-\left(l+1\right)^{-\left(1+\frac{1}{l}\right)p}\right)\\
 & \lesssim & \sum_{l=1}^{\infty}\dfrac{1}{l^{p}}.
\end{eqnarray*}
Therefore 
\begin{equation}
\sup\left\{ \left\Vert f_{j}\right\Vert _{L^{p}\left(\Omega_{k}\right)}^{p}:j\in\mathbb{Z}^{+}\right\} <c_{0}:=\sum_{l=1}^{\infty}\frac{1}{l^{p}}.\label{eq:Sha3}
\end{equation}

By construction, $f_{j}\in L^{2}\left(\Omega_{k}\right)$, for any $j\in\mathbb{Z}^{+}$. Thus we can make use of the above orthogonal basis $\mathcal{B}$
	of $A^{2}\left(\Omega_{k}\right)$ to obtain   
	\[
	T_{K^{-\alpha}}\left(f_{j}\right)=C\left(\intop_{a_{j+1}}^{1}\intop_{0}^{r_{2}^{\frac{1}{k}}}\left(1-r_{2}\right)^{2\alpha}\left(r_{2}-r_{1}^{k}\right)^{2\alpha}r_{1}r_{2}h\left(r_{2}\right)dr_{1}dr_{2}\right)\dfrac{1}{z_{2}},
	\]
	for a constant $C$ independent of $j.$ Therefore 
	\begin{equation}
	\left\Vert T_{K^{-\alpha}}\left(f_{j}\right)\right\Vert _{L^{q}\left(\Omega_{k}\right)}^{q}\gtrsim\sum_{l=1}^{j}\left(\intop_{a_{l+1}}^{a_{l}}r_{2}^{\frac{1}{l}-1+2\left(\alpha-\frac{k+1}{kp}+\frac{k+2}{2k}\right)}\left(1-r_{2}\right)^{2\alpha}dr_{2}\right).\label{eq:Sha2}
	\end{equation}
	
	We now  show that $\alpha>\frac{k+1}{kp}-\frac{k+2}{2k}$ by contradiction. Assume that $\alpha\le \frac{k+1}{kp}-\frac{k+2}{2k}$. Since 
	$\left(1-r_{2}\right)^{2\alpha}\geq\max\left\{ 1-2\alpha r_{2},2\alpha-2\alpha r_{2}\right\} $
	for any $\alpha\geq0$ and $r_{2}\in\left(0,1\right)$, 
	\eqref{eq:Sha2} continues as 
	\begin{eqnarray}
	\left\Vert T_{K^{-\alpha}}\left(f_{j}\right)\right\Vert _{L^{q}\left(\Omega_{k}\right)}^{q} & \gtrsim & \sum_{l=1}^{j}\left(\intop_{a_{l+1}}^{a_{l}}r_{2}^{\frac{1}{l}-1}dr_{2}\right)-\sum_{l=1}^{j}\left(\intop_{a_{l+1}}^{a_{l}}r_{2}^{\frac{1}{l}}dr_{2}\right).
	\label{eq:Sha4}
	\end{eqnarray}
	Note that the first sum of the RHS of \eqref{eq:Sha4} goes to infinity while the second sum converges as $j\rightarrow\infty$. The contradiction now follows from \eqref{eq:Sha1}, \eqref{eq:Sha3}
	and \eqref{eq:Sha4} by letting $j\rightarrow\infty$. 
\end{proof}

\bibliographystyle{apa}

\begin{thebibliography}{}
	
	\bibitem[\protect\astroncite{B\l~ocki}{2005}]{Blo05}
	B{\l}ocki, Z. (2005).
	\newblock The {B}ergman metric and the pluricomplex {G}reen function.
	\newblock {\em Trans. Amer. Math. Soc.}, 357(7):2613--2625.
	
	\bibitem[\protect\astroncite{B\l~ocki}{2014}]{Blo14}
	B{\l}ocki, Z. (2014).
	\newblock Cauchy-{R}iemann meet {M}onge-{A}mp\`ere.
	\newblock {\em Bull. Math. Sci.}, 4(3):433--480.
	
	\bibitem[\protect\astroncite{B{\l}ocki}{2014}]{Blok14}
	B{\l}ocki, Z. (2014).
	\newblock A lower bound for the {B}ergman kernel and the {B}ourgain-{M}ilman
	inequality.
	\newblock In {\em Geometric aspects of functional analysis}, volume 2116 of
	{\em Lecture Notes in Math.}, pages 53--63. Springer, Cham.
	
	\bibitem[\protect\astroncite{Carlehed et~al.}{1999}]{CCW99}
	Carlehed, M., Cegrell, U., and Wikstr\"om, F. (1999).
	\newblock Jensen measures, hyperconvexity and boundary behaviour of the
	pluricomplex {G}reen function.
	\newblock {\em Ann. Polon. Math.}, 71(1):87--103.
	
	\bibitem[\protect\astroncite{Chakrabarti and Shaw}{2013}]{ChSh13}
	Chakrabarti, D. and Shaw, M.-C. (2013).
	\newblock Sobolev regularity of the {$\overline{\partial}$}-equation on the
	{H}artogs triangle.
	\newblock {\em Math. Ann.}, 356(1):241--258.
	
	\bibitem[\protect\astroncite{Chakrabarti and Zeytuncu}{2016}]{ChZe16}
	Chakrabarti, D. and Zeytuncu, Y.~E. (2016).
	\newblock {$L^p$} mapping properties of the {B}ergman projection on the
	{H}artogs triangle.
	\newblock {\em Proc. Amer. Math. Soc.}, 144(4):1643--1653.
	
	\bibitem[\protect\astroncite{Charpentier and Dupain}{2006}]{ChDu06}
	Charpentier, P. and Dupain, Y. (2006).
	\newblock Estimates for the {B}ergman and {S}zeg\"o projections for
	pseudoconvex domains of finite type with locally diagonalizable {L}evi form.
	\newblock {\em Publ. Mat.}, 50(2):413--446.
	
	\bibitem[\protect\astroncite{Chen and Fu}{2011}]{ChFu11}
	Chen, B.-Y. and Fu, S. (2011).
	\newblock Comparison of the {B}ergman and {S}zeg\"o kernels.
	\newblock {\em Adv. Math.}, 228(4):2366--2384.
	
	\bibitem[\protect\astroncite{Chen}{2017}]{ChL17}
	Chen, L. (2017).
	\newblock The {$L^p$} boundedness of the {B}ergman projection for a class of
	bounded {H}artogs domains.
	\newblock {\em J. Math. Anal. Appl.}, 448(1):598--610.
	
	\bibitem[\protect\astroncite{Edholm}{2016}]{Edh16}
	Edholm, L.~D. (2016).
	\newblock Bergman theory of certain generalized {H}artogs triangles.
	\newblock {\em Pacific J. Math.}, 284(2):327--342.
	
	\bibitem[\protect\astroncite{Edholm and McNeal}{2016}]{EdMcN16}
	Edholm, L.~D. and McNeal, J.~D. (2016).
	\newblock The {B}ergman projection on fat {H}artogs triangles: {$L^p$}
	boundedness.
	\newblock {\em Proc. Amer. Math. Soc.}, 144(5):2185--2196.
	
	\bibitem[\protect\astroncite{Edholm and McNeal}{2017}]{EdMcN17}
	Edholm, L.~D. and McNeal, J.~D. (2017).
	\newblock Bergman subspaces and subkernels: degenerate {$L^p$} mapping and
	zeroes.
	\newblock {\em J. Geom. Anal.}, 27(4):2658--2683.
	
	\bibitem[\protect\astroncite{Harrington and Zeytuncu}{2017}]{HaZe17}
	Harrington, P. and Zeytuncu, Y. (2017).
	\newblock ${L}^p$ mapping properties for the {C}auchy-{R}iemann equations on
	{L}ipschitz domains admitting subelliptic estimates.
	\newblock {\em arXiv:1705.07374}.
	
	\bibitem[\protect\astroncite{Herbort}{1999}]{Her99}
	Herbort, G. (1999).
	\newblock The {B}ergman metric on hyperconvex domains.
	\newblock {\em Math. Z.}, 232(1):183--196.
	
	\bibitem[\protect\astroncite{Khanh et~al.}{2017}]{KLT17}
	Khanh, T., Liu, J., and Thuc, P. (2017).
	\newblock {B}ergman-{T}oeplitz operators on weakly pseudoconvex domains.
	\newblock {\em arXiv:1710.10761}.
	
	\bibitem[\protect\astroncite{Klimek}{1991}]{Kli91}
	Klimek, M. (1991).
	\newblock {\em Pluripotential theory}, volume~6 of {\em London Mathematical
		Society Monographs. New Series}.
	\newblock The Clarendon Press, Oxford University Press, New York.
	\newblock Oxford Science Publications.
	
	\bibitem[\protect\astroncite{Klimek}{1995}]{Kli95}
	Klimek, M. (1995).
	\newblock Invariant pluricomplex {G}reen functions.
	\newblock In {\em Topics in complex analysis ({W}arsaw, 1992)}, volume~31 of
	{\em Banach Center Publ.}, pages 207--226. Polish Acad. Sci. Inst. Math.,
	Warsaw.
	
	\bibitem[\protect\astroncite{McNeal and Stein}{1994}]{McNSt94}
	McNeal, J. and Stein, E. (1994).
	\newblock Mapping properties of the {B}ergman projection on convex domains of
	finite type.
	\newblock {\em Duke Math.\ J.}, 73:177--199.
	
	\bibitem[\protect\astroncite{McNeal}{1989}]{McN89}
	McNeal, J.~D. (1989).
	\newblock Boundary behavior of the {B}ergman kernel function in {${\bf C}^2$}.
	\newblock {\em Duke Math. J.}, 58(2):499--512.
	
	\bibitem[\protect\astroncite{Phong and Stein}{1977}]{PhSt77}
	Phong, D.~H. and Stein, E.~M. (1977).
	\newblock Estimates for the {B}ergman and {S}zeg\"o projections on strongly
	pseudo-convex domains.
	\newblock {\em Duke Math. J.}, 44(3):695--704.
	
	\bibitem[\protect\astroncite{Zwonek}{1999}]{ZWO99}
	Zwonek, W.~o. (1999).
	\newblock On {B}ergman completeness of pseudoconvex {R}einhardt domains.
	\newblock {\em Ann. Fac. Sci. Toulouse Math. (6)}, 8(3):537--552.
	
\end{thebibliography}

\end{document}